\documentclass[preprint,11pt]{elsarticle}

\usepackage[left=2cm,right=2.5cm]{geometry}

\usepackage{graphicx}
\usepackage{amssymb}
\usepackage{amsmath}
\usepackage{amsthm}
\usepackage{latexsym}
\usepackage{eufrak}
\usepackage{euscript}
\usepackage{exscale}
\usepackage{graphicx}
\usepackage{mathrsfs}
\usepackage{lineno}
\usepackage{hyperref}
\usepackage[table,xcdraw]{xcolor}
\usepackage{mathabx}

\newtheorem{theorem}{Theorem}[section]
\newtheorem{lemma}[theorem]{Lemma}
\newtheorem{cor}[theorem]{Corollary}

\theoremstyle{definition}
\newtheorem{definition}[theorem]{Definition}
\newtheorem{example}{Example}
\newtheorem{remark}[theorem]{Remark}

\newcommand{\kc}{\mathcal{C}}
\newcommand{\pc}{\mathcal{P}}

\journal{}

\begin{document}

\begin{frontmatter}



\title{\LARGE{On $d$-stochastic measures with fractal support and uniform $(d-1)$-marginals, and 
related results}}

\author[label1]{Nicolas Pascal Dietrich}
\ead{nicolaspascal.dietrich@plus.ac.at}
\author[label2]{Juan Fern\'andez S\'anchez}
\ead{juanfernandez@ual.es}
\author[label1]{Wolfgang Trutschnig\corref{cor1}}
\ead{wolfgang@trutschnig.net}
\affiliation[label1]{organization={University of Salzburg, Department for Artificial Intelligence and Human Interfaces},
            addressline={Hellbrunner Straße 34},
            city={Salzburg},
            postcode={5020},
            state={Salzburg},
            country={Austria}}
\affiliation[label2]{organization={Universidad de Almer\'a, 
                    Grupo de investigaci\'on de An\'alisis Matem\'atico},
                    addressline={La Canada de San Urbano},
                    country={Spain}}
	       
\cortext[cor1]{Corresponding Author}

\begin{abstract}
The family $\mathcal{P}_{d}^{\lambda_{d-1}}$ of all probability measures on 
$[0,1]^d$ whose $(d-1)$-dimensional marginals are all equal to the Lebesgue 
measure $\lambda_{d-1}$ on $[0,1]^{d-1}$ contains remarkably pathological elements:   
Working with Iterated Function Systems with Probabi\-lities (IFSPs) we 
construct measures $\mu \in \mathcal{P}_{d}^{\lambda_{d-1}}$ of the following two types: 
(i) $\mu$ has self-similar fractal support; (ii) $\mu$ has self-similar support and models the situation of complete/functional dependence in each direction.
As our main results concerning type (i) we prove, firstly, that for every 
$d\geq 3$ the 
set $\mathcal{D}_d$ of Hausdorff dimensions of the supports of elements in 
$\mathcal{P}_{d}^{\lambda_{d-1}}$ 
is dense in $[d-1,d]$; and, secondly, that the subset of elements in 
$\mathcal{P}_{d}^{\lambda_{d-1}}$  having fractal support is dense in 
$\mathcal{P}_{d}^{\lambda_{d-1}}$  with respect to the Wasserstein metric. 
Moreover, we show the existence of an element in $\mathcal{P}_{3}^{\lambda_{2}}$  of type (ii) 
whose support is a Sierpinski tetrahedron and study some generalizations. 
\end{abstract}




\end{frontmatter}

\section{Introduction}
A probability measure $\mu$ on the unit cube $[0,1]^d$ is called $d$-stochastic, if
all univariate marginal distributions of $\mu$ coincide with the uniform distribution 
on $[0,1]$. Considering that each $d$-stochastic measure $\mu$ corresponds to a unique
$d$-dimensional copula $A_\mu$ (and vice versa, see, e.g., \cite{DS}) and 
that copulas are Lipschitz continuous, it might seem natural to assume that $d$-stochastic
measures distribute their mass in a fairly regular way over $[0,1]^d$. 
About 20 years ago, in \cite{FNRL} Fredricks, Nelsen and Rodr\'iguez-Lallena (FNR) falsified this interpretation by proving 
the existence of a doubly stochastic measure $\mu_s$ whose support has Hausdorff dimension 
$s \in (1,2)$ for every $s$. Remarkably pathological, but mathematically interesting elements of the 
convex set $\mathcal{P}_d$ of $d$-stochastic measures, however, were already studied  
in the second half of the past century: In 1965, Lindestrauss \cite{Li} 
proved a conjecture by Phelps saying that extreme points of $\mathcal{P}_2$ are necessarily 
singular with respect to the Lebesgue measure $\lambda_2$, Losert \cite{Lo} and Feldman \cite{Fe} 
constructed an extreme point of $\mathcal{P}_2$ with full support. A full characterization of 
extreme points of $\mathcal{P}_2$, however, is still unknown, underlining 
the fact that $\mathcal{P}_2$ is far more complex than its discrete counterpart, the family of 
doubly stochastic $N \times N$ matrices. 

In the current paper we revisit the FNR result and study its strict extension to the family 
$\mathcal{Z}_d=\mathcal{P}_d^{\lambda_{d-1}}$ of $d$-stochastic measures 
whose $(d-1)$-dimensional marginals are all equal to the Lebesgue 
measure $\lambda_{d-1}$ on $[0,1]^{d-1}$ (the weak extension to $\mathcal{P}_d$ was established in \cite{TFS}). As our main contributions 
we show that the set $\mathcal{D}_d$ of Hausdorff dimensions of supports of elements in  $\mathcal{P}_d^{\lambda_{d-1}}$ is dense in $[d-1,d]$ and that elements with fractal support
are even dense in the metric space $(\mathcal{P}_d^{\lambda_{d-1}},h)$ with $h$ denoting the 
Wasserstein metric. Moreover, as a surprising by-product we prove the existence of measures 
$\mu \in \mathcal{P}_d^{\lambda_{d-1}}$ with self-similar support, which describe 
the situation of complete dependence in each direction, i.e., 
for a random vector $\mathbf{X}=(X_1,\ldots,X_d)$ having distribution $\mu$, 
each variable $X_j$ is a function of the remaining $(d-1)$ variables almost surely. 
As a special case, our construction produces a $3$-stochastic measure of the afore-mentioned 
type whose support is a Sierpinski tetrahedron (see Figure \ref{fig1} for an approximation
and \cite{Ts} for additional properties of the Sierpinski tetrahedron). 

\begin{figure}[h!]
    \centering
    \includegraphics[width=0.9\textwidth]{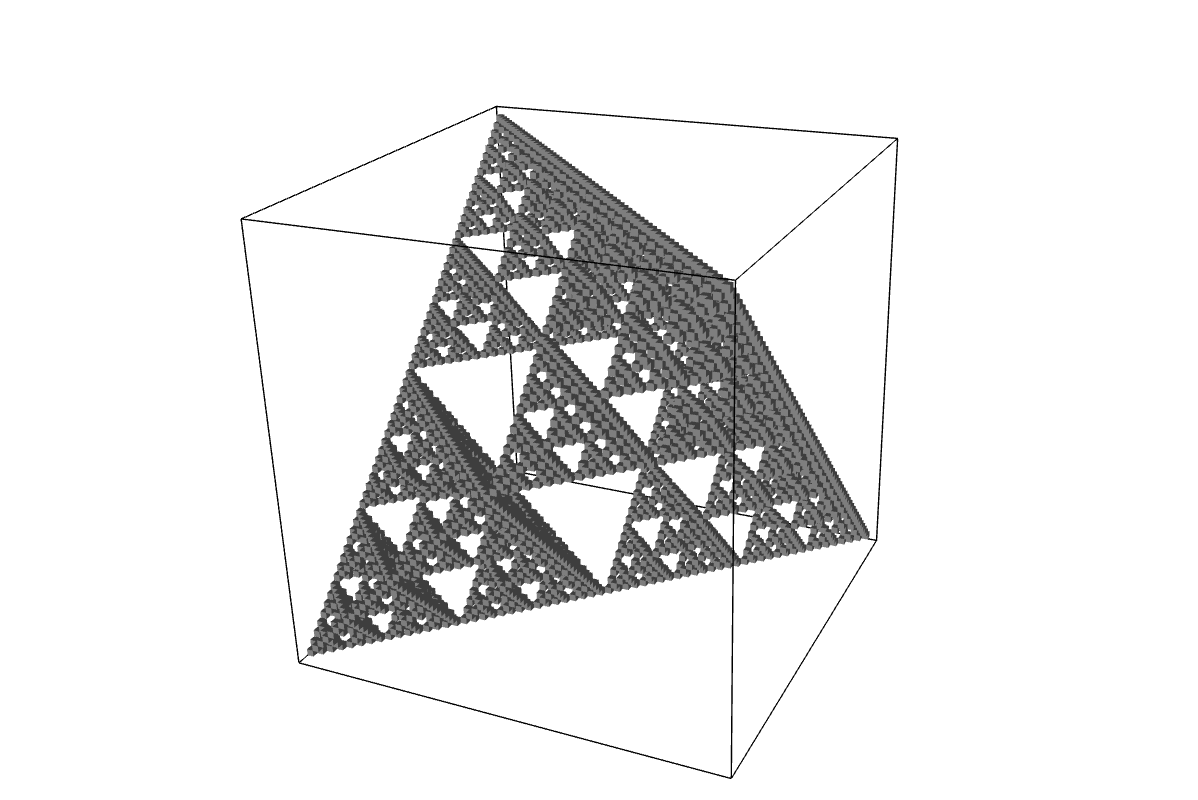}
    \caption{Approximation of the support of the $3$-stochastic measure $\mu$ 
    considered in Example \ref{ex:sierpinski}; in this case the support of $\mu$ is 
    a Sierpinski tetrahedron.}
    \label{fig1}
\end{figure}

\section{Notation and preliminaries}
Throughout the paper $d\geq 2$ will denote the dimension and 
bold symbols will denote vectors. For every fixed 
index/coordinate $i_0 \in \{1,\ldots,d\}$ and every vector $\mathbf{x} \in \mathbb{R}^d$ we will write
$\mathbf{x}_{-i_0}=(x_1,\ldots,x_{i_0-1},x_{i_0+1},\ldots,x_d)$. 
Moreover, for every vector $\mathbf{j}=(j_1,\ldots,j_l)$ of at most $l \leq d-1$ pairwise 
disjoint indices in $\{1,\ldots,d\}$ we will let $\pi_\mathbf{j}: [0,1]^d \rightarrow [0,1]^l$ 
denote the projection onto the coordinates in $\mathbf{j}$, i.e., 
$\pi_{\mathbf{j}}(x_1,\ldots,x_d)=(x_{j_1},\ldots,x_{j_l})$. 
To simplify notation we will also write $\pi_{1:l}$ for the projection onto the first $l$ coordinates
as well as $\pi_{-j_0}$ for the projection onto the coordinates $(1,\ldots,j_0-1,j_0+1,\ldots,d)$. 

For every metric space $(\Omega,\rho)$ we will let $\mathcal{K}(\Omega)$ denote the family of all 
non-empty compact subsets of $\Omega$, $\delta_H$ the Hausdorff metric on $\mathcal{K}(\Omega)$, and 
$\mathcal{B}(\Omega)$ the Borel $\sigma$-field. 
$\mathcal{P}(\Omega)$ denotes the family of all probability measures on 
$(\Omega,\mathcal{B}(\Omega))$. The Lebesgue measure on $[0,1]^d$ will be denoted by $\lambda_d$. \\ 
As mentioned in the introduction, a probability measure $\mu \in \mathcal{P}([0,1]^d)$ is called $d$-stochastic if the 
push-forward $\mu^{\pi_j}$ of $\mu$ via the projection $\pi_j: [0,1]^d \rightarrow [0,1]$ 
fulfills $\mu^{\pi_j}=\lambda_1$ for every $j \in \{1,\ldots,d\}$. 
Every $d$-stochastic measure $\mu$ corresponds to a unique $d$-dimensional copula 
$A_\mu$ and vice versa; the correspondence is established via 
$$
\mu\left(\bigtimes_{j=1}^d [0,x_j] \right) = A_\mu(\mathbf{x}), \quad 
\mathbf{x}=(x_1,\ldots,x_d) \in [0,1]^d.
$$
$\kc_d$ will denote the class of all $d$-dimensional copulas, $\mathcal{P}_{d}$ the family 
of all $d$-stochastic measures. For every $A \in \kc_d$
the corresponding $d$-stochastic measure will be denoted by $\mu_A$; for the product copula
$\Pi_d \in \kc_d$ we obviously have $\lambda_d=\mu_{\Pi_d}$. \\ 
For a random vector $\mathbf{X}=(X_1,\ldots,X_d)$ on a probability space 
$(\Omega,\mathcal{A},\mathbb{P})$ and $\mu \in \mathcal{P}_d$ we will write 
$\mathbf{X}\sim \mu$ if $\mu$ is the distribution of $\mathbf{X}$, i.e., if  $\mathbb{P}^\mathbf{X}=\mu$ holds. Notice that for 
$(X_1,\ldots,X_d)\sim \mu \in \mathcal{P}_d$ we have that each $X_i$ is uniformly 
distributed on $[0,1]$.   
For every $\mu \in \mathcal{P}_d$ and every vector 
$\mathbf{j}=(j_1,\ldots,j_l)$ of at most $l \leq d-1$ 
pairwise disjoint indices in $\{1,\ldots,d\}$ we will let $A^\mathbf{j}$ denote the marginal 
copula correspon\-ding to $\mathbf{j}$, i.e., the copula corresponding to the push-forward 
$\mu_{A^\mathbf{j}}=\mu_A^{\pi_{\mathbf{j}}}$ of $\mu_A$ via $\pi_{\mathbf{j}}$; 
$A^{1:l}$ denotes the marginal copula of the first $l$ coordinates and $A^{-j_0}$ the 
marginal copula corresponding to $\mu^{\pi_{-j_0}}$.   
For further properties of $d$-stochastic measures and copulas we refer to \cite{DS, Ne}. \\

We refer to a map 
$K: \mathbb{R}^{d-1}\times\mathcal{B}(\mathbb{R}) \rightarrow [0,1]$ as
$(d-1)$-\emph{Markov kernel} from 
$\mathbb{R}^{d-1}$ to $\mathbb{R}$, if the function 
$\mathbf{x}\mapsto K(\mathbf{x},E)$ is 
$\mathcal{B}(\mathbb{R}^{d-1})$-$\mathcal{B}(\mathbb{R})$-measurable for every fixed 
$E\in\mathcal{B}(\mathbb{R})$ and the map $E\mapsto K(\mathbf{x},E)$ is a probability measure 
on $\mathcal{B}(\mathbb{R})$ for every $\mathbf{x}\in\mathbb{R}^{d-1}$. 
Given a random variable $Y$ and a $(d-1)$-dimensional random vector 
$\mathbf{X}$ on a joint probability space $(\Omega,\mathcal{A},\mathbb{P})$, a Markov kernel 
$K$ is called a regular conditional distribution of
$Y$ given $\mathbf{X}$, if (and only if) for every set 
$E \in \mathcal{B}(\mathbb{R})$ the identity 
$$
K(\mathbf{X}(\omega), E) = \mathbb{E}(\mathbf{1}_E \circ Y | \mathbf{X})(\omega)
$$
holds for $\mathbb{P}$-almost every $\omega \in \Omega$.
It is a well-known fact (see \cite{Kallenberg}) that for each pair $(\mathbf{X}, Y)$ 
such a regular conditional distribution $K$ of $Y$ given $\mathbf{X}$ exists and that it is 
unique for $\mathbb{P}^{\mathbf{X}}$-almost every $\mathbf{x} \in \mathbb{R}^{d-1}$.\\
For $(\mathbf{X}, Y) \sim \mu$ we will let 
$K_{\mu}:[0,1]^{d-1} \times \mathcal{B}([0,1]) \to [0,1]$ denote (a version of) the corresponding 
conditional distribution of $Y$ given $\mathbf{X}$; $K_\mu$ will simply be referred to as 
(a version of) the $(d-1)$-\emph{Markov kernel} of the $d$-stochastic measure $\mu$ (or of the 
copula $A_\mu$). 
For every $G \in \mathcal{B}([0,1]^d)$ and $\mathbf{x} \in [0,1]^{d-1}$ 
define the $\mathbf{x}$-section $G_\mathbf{x}$ of $G$ by $G_{\mathbf{x}}:=\{y\in [0,1]: 
(\mathbf{x},y) \in G\}\in\mathcal{B}([0,1])$. 
Applying \emph{disintegration} of $\mu \in \mathcal{P}_d$ into the marginal 
$\mu^{\pi_{1:d-1}} $ and the $(d-1)$-Markov kernel $K_\mu$ of $A$
(see \cite[Section 5]{Kallenberg}) the following identity 
holds for all $G \in \mathcal{B}([0,1]^d)$:
\begin{align}\label{eq:DI}
	\mu(G) = \int_{[0,1]^{d-1}} K_{\mu}(\mathbf{x},G_{\mathbf{x}})
	\, \mathrm{d}\mu^{\pi_{1:d-1}}(\mathbf{x}).
\end{align}
Following \cite{GJT} we will call a measure $\mu \in \mathcal{P}_d$ or the corresponding 
copula $A \in \kc_d$ \emph{completely dependent} (on the first
$(d-1)$ coordinates), if
there exists some $\mu^{\pi_{1:d-1}}$-$\lambda$ preserving transformation $g$ (i.e., a measurable
transformation fulfilling that the push-forward 
$(\mu^{\pi_{1:d-1}})^g$ of $\mu^{\pi_{1:d-1}}$ via $g$ 
coincides with $\lambda$) such that
$K(\mathbf{x},F):=\mathbf{1}_F(g(\mathbf{x}))$ is a regular conditional distribution of $\mu$. 
For $(\mathbf{X},Y)\sim \mu \in \mathcal{P}_d$ we obviously have that complete dependence of 
$\mu$ is equivalent to the existence of some $\mu^{\pi_{1:d-1}}$-$\lambda$ preserving transformation $g$ such that $Y=g(\mathbf{X})$ almost surely. \\  

Finally we recall the definition of an Iterated Function System (IFS) and some main results about IFSs (for more details see \cite{Bar,Ed}). 
In what follows we assume that $(\Omega,\rho)$ is a compact metric space. 
We call a mapping $f: \Omega \rightarrow \Omega$ a \emph{contraction} if there 
exists some constant $L<1$ such that $\rho(f(x),f(y)) \leq L \rho(x,y)$ holds for all 
$x,y \in \Omega$. We will refer to a mapping $f: \Omega \rightarrow \Omega$ as
\emph{similarity}, if there exists some constant $c>0$ such that 
$\rho(f(x),f(y)) =c \rho(x,y)$ for all $x,y \in \Omega$.
A family $(f_l)_{l=1}^m$ of $m\geq 2$ contractions on $\Omega$ is called 
\emph{Iterated Function System} (IFS for short) and will be denoted by $\{\Omega, (f_l)_{l=1}^m\}$. 
An IFS together with a vector $(p_l)_{l=1}^m \in (0,1]^m$ fulfilling $\sum_{l=1}^m p_l=1$ 
is called \emph{Iterated Function System with probabilities} (IFSP for short). 
We will denote IFSPs by $\{\Omega,(f_l)_{l=1}^m,(p_l)_{l=1}^m \}$.  
Every IFSP induces the so-called \emph{Hutchinson operator} 
$\mathcal{H}: \mathcal{K}(\Omega) \rightarrow  \mathcal{K}(\Omega)$, defined by  
\begin{equation}\label{Hutch}
\mathcal{H}(Z):=\bigcup_{i=1}^m \,f_i(Z).
\end{equation}
It can be shown (see \cite{Bar}) that $\mathcal{H}$ is a contraction on the 
compact metric space $(\mathcal{K}(\Omega),\delta_H)$, so 
Banach's Fixed Point theorem implies the existence of a unique, globally attractive fixed point $Z^\star$ of $\mathcal{H}$, 
i.e., for every $R \in \mathcal{K}(\Omega)$ we have 
$$
\lim_{n \rightarrow \infty} \delta_H\big(\mathcal{H}^n(R),Z^\star\big)=0.
$$
If the contractions of the IFSP are similarities, then the fixed point 
$Z^\star$ is self-similar, i.e., $Z^\star$ is the union of shrunk copies of itself. 
Moreover, in the special case of $\Omega=\mathbb{R}^d$, if the IFSP only contains 
similarities and fulfills the so-called open set condition (see \cite{Fal}), then 
the Hausdorff dimension of $Z^\star$ fulfills $\textrm{dim}_H(Z^\star)=s$ where $s$ satisfies 
\begin{eqnarray}\label{eq:calc_hausdorffdim}
    \sum_{i=1}^m c_i^s=1
\end{eqnarray}
and $c_i$ denotes the shrinking factor of similarity $f_i$.  
On the other hand, every IFSP $\{\Omega,(f_l)_{l=1}^m,(p_l)_{l=1}^m \}$ also induces a so-called \emph{Markov operator} 
$\mathcal{V}:\mathcal{P}(\Omega) \rightarrow \mathcal{P}(\Omega)$, defined by 
($\mu^{f_i}$ denoting the push-forward of $\mu$ via $f_i$)
\begin{equation}\label{HM}
\mathcal{V}(\mu):=\sum_{i=1}^m p_i\,\mu^{f_i}
\end{equation}
for every $\mu \in \mathcal{P}(\Omega)$.
The so-called \emph{Hutchison metric} $h$ (a.k.a. Kantorovich-Rubinstein and Wasserstein 
$1$-metric) on $\mathcal{P}(\Omega)$ is defined by 
\begin{equation}\label{hutchm}
h(\mu,\nu):= \sup \bigg\{\int_\Omega g\,d\mu - \int_\Omega  g\,d\nu:\, g \in Lip_1(\Omega,\mathbb{R}) \bigg\}, 
\end{equation}
whereby $Lip_1(X,\mathbb{R})$ is the class of all non-expanding functions $g:\Omega \rightarrow \mathbb{R}$, i.e., 
functions fulfilling $\vert g(x)-g(y) \vert \leq \rho(x,y)$ for all $x,y \in \Omega$. 
It is not difficult to show that $\mathcal{V}$ is a contraction on $(\mathcal{P}(\Omega),h)$, that 
$h$ is a metrization of the topology of weak convergence on $\mathcal{P}(\Omega)$,  
and that $(\mathcal{P}(\Omega),h)$ is a compact metric space (see \cite{Bar,Dud}). Consequently, again
by Banach's Fixed Point theorem, it follows that there is a unique, globally attractive fixed point
$\mu^\star \in \mathcal{P}(\Omega)$ of $\mathcal{V}$, i.e., for every $\nu \in \mathcal{P}(\Omega)$ 
we have
\begin{equation}\label{eq:mustar.global}
\lim_{n \rightarrow \infty} h\big(\mathcal{V}^n(\nu),\mu^\star\big)=0.
\end{equation}

The fixed point $Z^\star$ and the measure $\mu^\star$ are closely related - in fact, 
the support $\textrm{supp}(\mu^\star)$ of $\mu^\star$ is $Z^\star$ (see \cite{Bar,Ku}).

\section{Singular $d$-stochastic measures with uniform $(d-1)$-dimensional marginals}
We start by recalling the construction of $d$-stochastic measures with fractal support via so-called 
generalized transformation matrices going back to 
\cite{TFS}. Consider the dimension $d\geq 2$, let $\mathbf{m}:=(m_1,\ldots,m_d) \in \mathbb{N}^d
$ be arbitrary but fixed, 
set $I_i=\{1,\ldots,m_i\}$ for every $i \in \{1,\ldots,d\}$, and define
\begin{equation}\label{ID}
\mathcal{I}_d^\mathbf{m}:=\bigtimes_{i=1}^d \{1,\ldots,m_i\}. \,\,\, 
\end{equation}
We will denote elements in $\mathcal{I}_d^\mathbf{m}$ in the 
form $\mathbf{i} =(i_1,\ldots,i_d)$, and, for every
probability distribution $\tau$ on $(\mathcal{I}_d^\mathbf{m},2^{\mathcal{I}_d^\mathbf{m}})$ 
write $\tau(\mathbf{i}):=\tau(\{\mathbf{i}\})$ for the point mass in $\mathbf{i}$.
\begin{definition}[Extended version of \cite{TFS}]\label{trafod}
Suppose that $d\geq 2$, that $\mathbf{m} \in \mathbb{N}^d$ fulfills $\min_{j} m_j \geq 2$, and let 
$\mathcal{I}_d^\mathbf{m}$ be defined according to eq. (\ref{ID}). 
A probability distribution $\tau$ on $(\mathcal{I}_d^\mathbf{m}, 2^{\mathcal{I}_d^\mathbf{m}})$ 
is called 
\emph{generalized transformation matrix} if for every $j \in \{1,\ldots,d\}$ 
and every $k \in \{1,\ldots,m_j\}$
$$
\tau\left(I_1 \times I_2 \times \ldots \times I_{j-1} \times \{k\} \times I_{j+1} \times 
\ldots \times  I_d \right)=\sum_{\mathbf{i}\in \mathcal{I}_d^\mathbf{m}: \,\,i_j=k}\tau(\mathbf{i}) >0
$$
holds. $\mathcal{S}_{\tau} \subseteq \mathcal{I}_d^\mathbf{m}$ will denote the support of $\tau$, i.e., 
$$
\mathcal{S}_{\tau}:=\{\mathbf{i}\in \mathcal{I}_d^\mathbf{m}: \, \tau(\mathbf{i})>0\}.
$$
Moreover, $\mathcal{T}_d^\mathbf{m}$ will denote the family of all $d$-dimensional 
generalized transformation matrices on $\mathcal{I}_d^\mathbf{m}$.   
\end{definition}
Every $\tau \in \mathcal{T}_d^\mathbf{m}$ induces a partition of $[0,1]^d$ in the following way: 
For each $ j \in \{1,\ldots,d\}$ define $a^j_0:=0$,  
\begin{equation}\label{eq:col.sums}
a^j_k:=\sum_{\mathbf{i}\in \mathcal{I}_d^\mathbf{m}: \,\,i_j\leq k}\tau(\mathbf{i}), 
\quad E^j_k:=[a^j_{k-1}, a^j_k],
\end{equation}
for every $k \in I_j$. Then $\bigcup_{k \in I_j} E^j_k = [0,1]$ and 
$E^j_{k_1} \cap E^j_{k_2}$ is empty or consists of exactly one point whenever $k_1\not = k_2$. Setting
$R_{\mathbf{i}}:=\times_{j=1}^d E^j_{i_j}$ for every $\mathbf{i} \in \mathcal{I}_d^\mathbf{m}$
therefore yields a family of compact rectangles
$(R_{\mathbf{i}})_{\mathbf{i}\in \mathcal{I}_d^\mathbf{m}}$ whose union is $[0,1]^d$ and which additionally fulfills that
$R_{\mathbf{i}_1} \cap R_{\mathbf{i}_2}$ is empty or a set of $\lambda_d$-measure zero whenever
$\mathbf{i}_1 \not = \mathbf{i}_2$. \\
Moreover, every $\tau \in \mathcal{T}_d^\mathbf{m}$ induces affine contractions 
$f_{\mathbf{i}}:[0,1]^d \rightarrow R_{\mathbf{i}}$, given by
\begin{displaymath}
f_{\mathbf{i}}(x_1,\ldots,x_d) = 
\left( \begin{array}{c}
 a^1_{i_1-1} \\
 a^2_{i_2-1} \\
 \vdots \\
 a^d_{i_d-1}
\end{array} \right) \, + \, 
\left( \begin{array}{c}
 (a^1_{i_1} - a^1_{i_1-1}) \, x_1\\
 (a^2_{i_2} - a^2_{i_2-1}) \, x_2\\
 \vdots \\
 (a^d_{i_d} - a^d_{i_d-1}) \, x_d\\
\end{array} \right).
\end{displaymath}
Since the $j$-th coordinate of $f_{\mathbf{i}}(x_1,\ldots,x_d)$ only depends on $i_j$ and $x_j$ 
we will also denote it by $f^j_{i_j}$, i.e., $f^j_{i_j}:[0,1] \rightarrow E^j_{i_j}$, $f^j_{i_j}(x_j):= a^j_{i_j-1} \,+\,(a^j_{i_j} - a^j_{i_j-1}) \, x_j$.
It follows directly from the construction that 
\begin{equation}\label{ifsdefmulti}
\Big\{[0,1]^d, (f_\mathbf{i})_{\mathbf{i} \in \mathcal{S}_\tau},\tau(\mathbf{i})_{\mathbf{i} \in 
\mathcal{S}_\tau}\Big\}
\end{equation}
is an IFSP. Moreover it is straightforward to verify that this IFSP fulfills the 
open set condition. According to \cite{TFS} the Markov operator 
$\mathcal{V}_{\tau}: \mathcal{P}([0,1]^d) \rightarrow \mathcal{P}([0,1]^d)$, given by 
\begin{equation}\label{eq:markov}
\mathcal{V}_{\tau}(\mu) = \sum_{\mathbf{i} \in \mathcal{S}_\tau} \tau(\mathbf{i})\, \mu^{f_\mathbf{i}}
\end{equation}
maps $\mathcal{P}_d$ into $\mathcal{P}_d$, so we 
can also view $\mathcal{V}_{\tau}$ as a transformation mapping $\kc_d$ to $\kc_d$ 
and write $\mathcal{V}_{\tau}(A)$ for every $A \in \kc_d$. 
Considering that $\mathcal{P}_d$ is a closed metric space 
(again see \cite{TFS}) it follows that the unique fixed point $\mu_\tau^\star$ of 
$\mathcal{V}_{\tau}$ is an element of $\mathcal{P}_d$ and as such 
corresponds to a unique copula which we will denote by $A_\tau^\star \in \kc_d$, i.e., 
we have $\mu_{A_\tau^\star}=\mu_\tau^\star$.  \\
 
For the rest of the paper we will consider $d\geq 3$. 
We now show, how additional properties of $\tau$ translate to properties of $A_\tau^\star$. 
In particular, we will formulate a sufficient condition for $\tau$ assuring that the 
induced operator $\mathcal{V}_\tau$ preserves uniform $(d-1)$-dimensional marginals.  
Doing so we will work with the following definition: 
\begin{definition}
We say that $\tau \in \mathcal{T}_d^\mathbf{m}$ with $d\geq 3$ fulfills the uniformity condition w.r.t. coordinate $j_0 \in \{1,\ldots,d\}$, if 
\begin{equation}\label{eq:cond.tau}
\sum_{l=1}^{m_{j_0}} \tau((i_1,\ldots,i_{j_0-1},l,i_{j_0+1},\ldots,i_d)) = \prod_{j\neq j_0} \lambda(E_{i_j}^j) = 
\lambda_{d-1}\left(\bigtimes_{j\neq j_0}E_{i_j}^j \right)
\end{equation}
holds for all $\mathbf{i}_{-j_0} \in \bigtimes_{i\neq j_0} \{1,\ldots,m_i\}$. 
\end{definition}
IFSPs induced by generalized transformation matrices fulfilling the uniformity condition 
preserve uniform marginals - the following result holds:
\begin{lemma}\label{thm:preserve.marginal}
Suppose that $\mu \in \mathcal{P}_d$ fulfills $\mu^{\pi_{1:d-1}}=\lambda_{d-1}$ and that 
$\tau \in  \mathcal{T}_d^\mathbf{m}$ fulfills the uniformity condition
w.r.t. coordinate $d$. Then $(\mathcal{V}_\tau(\mu))^{\pi_{1:d-1}}=\lambda_{d-1}$.
\end{lemma} 
\begin{proof}
We show that for every rectangle $G=\bigtimes_{i=1}^{d-1} G_i \in \mathcal{B}([0,1]^{d-1})$
we have 
\begin{equation}\label{eq:marginal}
(\mathcal{V}_\tau(\mu))\left(\bigtimes_{j=1}^{d-1} G_j \times [0,1]\right) = 
\lambda_{d-1} \left( \bigtimes_{j=1}^{d-1} G_j \right)
\end{equation}
and proceed as follows: The left side of eq. (\ref{eq:marginal}) 
can be expressed as 
\begin{align*}
(\mathcal{V}_\tau(\mu))\left(\bigtimes_{j=1}^{d-1} G_j \times [0,1]\right) &= 
\sum_{\mathbf{i} \in \mathcal{I}_d^\mathbf{m}} \tau(\mathbf{i}) \,\mu^{f_{\mathbf{i}}} 
    \left(\bigtimes_{j=1}^{d-1} G_j \times [0,1]\right) \\
    &= \sum_{\mathbf{i} \in \mathcal{I}_d^\mathbf{m}} \tau(\mathbf{i}) \,\mu 
    \left(f_{\mathbf{i}}^{-1} \left(\bigtimes_{j=1}^{d-1} G_j \times [0,1]\right)\right) \\
    &= \sum_{\mathbf{i} \in \mathcal{I}_d^\mathbf{m}} \tau(\mathbf{i}) \,\mu 
    \left(\bigtimes_{j=1}^{d-1} (f_{i_j}^j)^{-1}(G_j) \times [0,1]\right) \\
    &= \sum_{\mathbf{i} \in \mathcal{I}_d^\mathbf{m}} \tau(\mathbf{i}) \,\mu^{\pi_{1:d-1}} 
    \left(\bigtimes_{j=1}^{d-1} (f_{i_j}^j)^{-1}(G_j) \right).
 \end{align*}
Hence, using the fact that by assumption $\mu^{\pi_{1:d-1}}=\lambda_{d-1}$ holds, it 
altogether follows that 
\begin{align*}
(\mathcal{V}_\tau(\mu))\left(\bigtimes_{j=1}^{d-1} G_j \times [0,1]\right)
&= \sum_{\mathbf{i} \in \mathcal{I}_d^\mathbf{m}} \tau(\mathbf{i}) \,\lambda_{d-1} 
    \left(\bigtimes_{j=1}^{d-1} (f_{i_j}^j)^{-1}(G_j) \right) = 
    \sum_{\mathbf{i} \in \mathcal{I}_d^\mathbf{m}} \tau(\mathbf{i}) \,\prod_{j=1}^{d-1} \lambda^{f_{i_j}^j}(G_j) \\
    &= \sum_{i_1 \in I_1, \ldots, i_{d-1} \in I_{d-1}} \, \sum_{k=1}^{m_d} \, 
    \tau(i_1,\ldots,i_{d-1},k) \,\prod_{j=1}^{d-1} \lambda^{f_{i_j}^j}(G_j) \\
    &= \sum_{i_1 \in I_1, \ldots, i_{d-1} \in I_{d-1}} \,  \, \prod_{j=1}^{d-1} \lambda(E_{i_j}^j)
    \,\prod_{j=1}^{d-1} \lambda^{f_{i_j}^j}(G_j) \\
    &= \lambda_{d-1} \left(\bigtimes_{j=1}^{d-1} G_j \right),
\end{align*} 
whereby the last equality follows from the transformation formula for 
the Lebes\-gue measure under affine transformations. 
As a direct consequence, the probability measures $(\mathcal{V}_\tau(\mu))^{\pi_{1:d-1}}$
and $\lambda_{d-1}$ coincide on the family of all measurable rectangles. Since the 
latter constitutes a generator of the Borel $\sigma$-field $\mathcal{B}([0,1]^{d-1})$, the proof 
is complete.
\end{proof} 
Proceeding analogously for every other coordinate $j_0 \in \{1,\ldots,d\}$ yields the 
following general result: 
\begin{theorem}\label{thm:main1}
Suppose that $\tau \in \mathcal{T}_d^\mathbf{m}$ fulfills the uniformity condition w.r.t. coordinate $j_0 \in \{1,\ldots,d\}$.
Then $(\mathcal{V}_\tau(\mu))^{\pi_{-j_0}}=\lambda_{d-1}$ if $\mu^{\pi_{-j_0}}=\lambda_{d-1}$. \\
Moreover, if for every $j_0 \in \{1,\ldots,d\}$ we have that (i) $\tau$ fulfills the uniformity 
condition w.r.t. coordinate $j_0$ and that (ii) $\mu^{\pi_{-j_0}}= \lambda_{d-1}$ holds, then   
all $(d-1)$-dimensional marginals of $\mathcal{V}_\tau(\mu)$ coincide with $\lambda_{d-1}$.  
\end{theorem} 
As mentioned in the introduction, in what follows we will let $\mathcal{P}_d^{\lambda_{d-1}}$ denote the family of 
all $d$-stochastic measures for which all $(d-1)$-dimensional marginals coincide with 
$\lambda_{d-1}$; $\kc_d^{\Pi_{d-1}}$ will denote the corresponding family of all $d$-dimensional copulas. Theorem \ref{thm:main1} has the following consequence.
\begin{theorem}\label{thm:main2}
Suppose that $\tau \in  \mathcal{T}_d^\mathbf{m}$ fulfills the uniformity condition 
w.r.t. every coordinate.
Then all $(d-1)$-dimensional marginals of $\mu_\tau^\star$ are $\lambda_{d-1}$, i.e., 
$\mu_\tau^\star \in \mathcal{P}_d^{\lambda_{d-1}}$. 
\end{theorem}
\begin{proof}
According to eq. (\ref{eq:mustar.global}) we know that for every $\mu \in \mathcal{P}_d$ 
the sequence $(\mathcal{V}_\tau^n(\mu))_{n \in \mathbb{N}}$ converges w.r.t. the Hutchinson 
metric $h$ to the unique fixed point $\mu_\tau^\star$. 
Considering $\mu=\lambda_d$, according to Theorem \ref{thm:main1} we have
$\mathcal{V}_\tau^n(\lambda_d) \in \mathcal{P}_d^{\lambda_{d-1}}$ for every $n \in \mathbb{N}$. 
Since $h$ is a metrization of weak convergence and weak convergence of a sequence of 
proba\-bility measures on $[0,1]^d$ implies weak convergence of all marginals, 
$\mu_\tau^\star \in \mathcal{P}_d^{\Pi_{d-1}}$ follows immediately.   
\end{proof}
\begin{cor}\label{thm:main3}
Suppose that $\tau \in  \mathcal{T}_d^\mathbf{m}$ fulfills the uniformity condition w.r.t. each coordinate and that 
there exists at least one $\mathbf{i} \in \mathcal{I}_d^\mathbf{m}$ with 
$\tau(\mathbf{i})=0$.
Then we have $\lambda_d(\textrm{supp}(\mu_\tau^\star))=0$, $\mu_\tau^\star$ is singular 
w.r.t. $\lambda_d$, and $\mu_\tau^\star \in \mathcal{P}_d^{\lambda_{d-1}}$.  
\end{cor}
\begin{proof}
Since $\textrm{supp}(\mu_\tau^\star)$ coincides with the fixed point $Z_\tau^\star$ of the 
Hutchinson operator 
$\mathcal{H}_\tau: \mathcal{K}([0,1]^d) \rightarrow \mathcal{K}([0,1]^d)$, given 
by
$$
\mathcal{H}_\tau(Z) = \bigcup_{\mathbf{i} \in \mathcal{S}_\tau} f_\mathbf{i}(Z),
$$ 
it suffices to show $\lambda_d(Z_\tau^\star)=0$, which can be easily established as follows: 
Consi\-dering $Z=[0,1]^d$ we obviously have that the sequence 
$(\mathcal{H}^n_\tau([0,1]^d))_{n \in \mathbb{N}}$ in $\mathcal{K}([0,1]^d)$ is monotonically 
decreasing to $Z_\tau^\star \in \mathcal{K}([0,1]^d)$. Letting $\mathbf{i}$ denote an 
element of $\mathcal{I}_d$ with $\tau(\mathbf{i})=0$ and setting 
$q=\prod_{j=1}^d \lambda(E_{i_j}^j) \in (0,1)$, it is straightforward to verify that 
$\lambda_d(\mathcal{H}^n_\tau([0,1]^d))\leq (1-q)^n$ holds for every 
$n \in \mathbb{N}$ (in each step, the volume shrinks at least by the factor 
$1-q$). Having that, using 
$$
Z_\tau^\star= \bigcap_{n=1}^\infty \mathcal{H}^n_\tau([0,1]^d)
$$
and the fact that $\lambda_d$ is continuous from above, directly yields  
$\lambda_d(\textrm{supp}(\mu_\tau^\star))=0$. Since the latter implies that 
$\mu_\tau^\star$ is singular w.r.t. $\lambda_d$, the proof is complete.
\end{proof}
Building on the previous general results, in the next section we study elements in 
$\mathcal{P}_d^{\lambda_{d-1}}$ describing the situation of full predictability, 
one of them being a $3$-stochastic measure whose support is 
a Sierpinski tetrahedron. 

\section{A completely dependent $3$-stochastic Sierpinski tetrahedron measure and some 
generalizations}
Despite mere existence might be surprising, it is not difficult to show that 
the class $\pc_3^{\lambda_2}$ and, more generally,
the family $\pc_d^{\lambda_{d-1}}$ contains completely dependent elements, i.e., 
elements describing the exact opposite of independence. 
In fact, considering $\mathbf{X}=(X_1,\ldots,X_{d-1}) \sim \lambda_{d-1}$ and 
defining $X_d:=\sum_{i=1}^{d-1} X_i \,\, (\textrm{mod(1)})$, it is, firstly, straightforward 
to verify that $X_d$ is uniformly distributed on $[0,1]$; and secondly, that every $X_j$ is 
a function of the other $(d-1)$ variables. In other words, the distribution $\mu$ 
of $(X_1,\ldots,X_d)$ is an element of $\pc_d^{\lambda_{d-1}}$. 
We are convinced that this simple but striking example already 
exists in the literature, but we have not been able to find any reference. \\
In the rest of this section we now show, how the IFSP approach can be used 
to construct other examples of the afore-mentioned type with the additional property
that the support of the $d$-stochastic measure $\mu \in \pc_d^{\lambda_{d-1}}$ 
is concentrated on a self-similar set. 
We first derive a general result for arbitrary $d \geq 3$ and then focus on $d=3$ 
to show the existence of a completely dependent $3$-stochastic measure with uniform two-dimensional marginals whose self-similar support is a Sierpinski tetrahedon. 

Minimizing technical complexity and keeping the notation as simple as possible, in 
the following we will mainly work with generalized transformation matrices fulfilling
additional uniformity conditions:
\begin{definition}\label{def:U}
Suppose that $d,N \geq 2$ and set $\mathbf{m}=(N,\ldots,N)$.
Then $\mathcal{U}^N_d$ will denote the family of all $\tau \in \mathcal{T}_d^\mathbf{m}$  
satisfying the following conditions:
\begin{itemize}
    \item[(i)] for every $j\in \{1,\ldots,d\}$ and every $k \in \{1,\ldots,N\}$ 
        we have $E^j_k=[\frac{k-1}{N},\frac{k}{N}]$. 
    \item[(ii)] $\tau$ fulfills the uniformity condition with respect to every coordinate.
\end{itemize}
\end{definition}
Notice that for $\tau \in \mathcal{U}_d^N$ the sets $R_\mathbf{i}$ are squares with side 
length $\frac{1}{N}$ and the uniformity condition (\ref{eq:cond.tau}) for coordinate $j_0$ 
boils down to 
\begin{equation}\label{eq:cond.tau2}
\sum_{l=1}^{N} \tau((i_1,\ldots,i_{j_0-1},l,i_{j_0+1},\ldots,i_d)) = \tfrac{1}{N^{d-1}}
\end{equation}
for every $\mathbf{i}_{-j_0} \in \{1,\ldots,N\}^{d-1}$. 
The following lemma collects some properties of the class $\mathcal{U}^N_d$ that will
be used in the sequel.
\begin{lemma}\label{lem:U}
    For all $d,N \geq 2$ the following assertions hold: 
    \begin{enumerate}
        \item $\mathcal{U}^N_d$ is convex.
        \item For every $\tau \in \mathcal{U}^N_d$ the induced IFSP only contains 
        similarities. 
        \item For every fixed $\tau \in \mathcal{U}^N_d$ and arbitrary permutations $\epsilon_1,\ldots,\epsilon_d$ of $\{1,\ldots,N\}$ we have that the probability measure $\tau'$ on $\{1,\ldots,N\}^d$, defined by
        $$
        \tau'((i_1,\ldots,i_d)):=\tau(\epsilon_1(i_1),\ldots,\epsilon_d(i_d)), 
        $$
        is an element of $\mathcal{U}^N_d$ too.
    \end{enumerate}
\end{lemma}
\begin{proof}
    Since convexity of $\mathcal{U}^N_d$ is obvious and the fact that each contraction 
    $f_\mathbf{i}$ is a similarity is a direct consequence of property (i) in 
    Definition \ref{def:U},     it suffices to prove the third assertion.\\
    Using bijectivity of the permutations $\epsilon_1,\ldots,\epsilon_d$, 
    the fact that $\tau'$ is a generalized transformation matrix is a direct consequence of 
    \begin{eqnarray*}
        \sum_{\mathbf{i}\in \mathcal{I}_d^\mathbf{m}: \,\,i_j=k}\tau'(\mathbf{i}) &=& 
        \sum_{\mathbf{i}\in \mathcal{I}_d^\mathbf{m}: \,\,i_j=k}
            \tau\left(\epsilon_1(i_1),\ldots,\epsilon_{j-1}(i_{j-1}),\epsilon_j(k),
            \epsilon_{j+1}(i_{j+1}),\ldots,\epsilon_d(i_d)\right) \\
         &=& \sum_{\mathbf{i}\in \mathcal{I}_d^\mathbf{m}: \,\,i_j=k}
            \tau\left(i_1,\ldots,i_{j-1},\epsilon_j(k),
            i_{j+1},\ldots,i_d\right) \\
         &=& a^j_{\epsilon_j(k)}- a^j_{\epsilon_j(k)-1}=\tfrac{1}{N} >0.
    \end{eqnarray*}
    The latter also directly yields property (i) in Definition \ref{def:U} for 
    $\tau'$. Moreover, for arbitrary but fixed 
    $i_1,i_2,\ldots,i_{j-1},i_{j+1},\ldots,i_{d-1} \in \{1,\ldots,N\}$, using
    $\tau \in \mathcal{U}^N_d$ and setting $s:=\sum_{l=1}^{N} \tau'((i_1,\ldots,i_{j_0-1},l,i_{j_0+1},\ldots,i_d)) $ we have
    \begin{eqnarray*}
   s&=& \sum_{l=1}^{N} \tau((\epsilon_1(i_1),\ldots,\epsilon_{j_0}(i_{j_0-1}),\epsilon_{j_0}(l),
    \epsilon_{j_0+1}(i_{j_0+1}),\ldots,\epsilon_d(i_d))) \\
    &=& \sum_{k=1}^{N} \tau((\epsilon_1(i_1),\ldots,\epsilon_{j_0}(i_{j_0-1}),k,
    \epsilon_{j_0+1}(i_{j_0+1}),\ldots,\epsilon_d(i_d))) = \tfrac{1}{N^{d-1}},
    \end{eqnarray*}
    i.e., $\tau'$ fulfills the uniformity condition w.r.t. every coordinate 
    $j_0 \in \{1,\ldots,d\}$.
\end{proof}


\begin{theorem}\label{thm:sierpiski.multi}
Let $d \geq 3$ as well as $N \geq 2$ be arbitrary but fixed,    
and suppose that the probability measure $\tau$ on $\{1,\ldots,N\}^d$ is given by   
\begin{equation}\label{eq:sierpsinski.multi}
\tau(\mathbf{i})=
\begin{cases}
\frac{1}{N^{d-1}} & \text{if } \sum_{j=1}^d (i_j-1) \, \in \{0,N,2N,\ldots,\}=:N\cdot \mathbb{N}_0\\
0 & \text{otherwise.}  
\end{cases}
\end{equation}
Then $\tau \in \mathcal{U}_d^N$ and the resulting measure $\mu_\tau^\star$ has the following properties:
\begin{itemize}
    \item[(P1)] $\mu_\tau^\star$ is an element of $\pc_d^{\lambda_{d-1}}$,
    \item[(P2)] the measure $\mu_\tau^\star$ is singular w.r.t. $\lambda_d$ 
    and has a self-similar set with Hausdorff dimension $d-1$ as support,     
    \item[(P3)] for $\mathbf{X}=(X_1,\ldots,X_d) \sim \mu_\tau^\star$ we have that 
each variable $X_j$ is almost surely a function of the other $(d-1)$ variables.  
\end{itemize}
\end{theorem}
\begin{proof}
We start with proving $\tau \in  \mathcal{U}_d^N$.  
First of all notice that $\tau$ according to eq. (\ref{eq:sierpsinski.multi}) obviously is 
permutation-invariant, i.e., for every permutation $\epsilon$ of $\{1,\ldots,d\}$ 
we have that $\tau(\epsilon(\mathbf{i}))=\tau(\mathbf{i})$ holds for all 
$\mathbf{i} \in \{1,\ldots,N\}^d$. As a direct consequence, in order to show that 
$\tau$ fulfills the uniformity condition w.r.t. every coordinate it suffices to
show uniformity w.r.t. the last coordinate, which can be done as follows:  
By construction, for arbitrary $(i_1,\ldots,i_{d-1}) \in \{1,\ldots,N\}^{d-1}$ there exist a unique 
index $l_0=l_0((i_1,\ldots,i_{d-1})) \in \{1,\ldots,N\}$ such that 
$\tau((i_1,\ldots,i_{d-1},l_0))>0$ holds, which 
directly yields
\begin{equation}
\sum_{l=1}^N \tau((i_1,\ldots,i_{d-1},l)) = \tfrac{1}{N^{d-1}}.
\end{equation}
As a direct consequence, for every fixed $i_1 \in \{1,\ldots,N\}$ we get
$$
\sum_{i_2,\ldots,i_{d-1},l=1}^N \tau((i_1,\ldots,i_{d-1},l)) = \tfrac{N^{d-2}}{N^{d-1}} 
= \tfrac{1}{N},
$$
which implies that the intervals $E^j_k=[a^j_{k-1}, a^j_k]$ do 
not depend on $j \in \{1,\ldots,d\}$ and fulfill 
$E^j_1=[0,\frac{1}{N}], E^j_2=[\frac{1}{N},\frac{2}{N}],\ldots,E^j_N=[\frac{N-1}{N},1]$. 
This shows that $\tau \in \mathcal{U}_d^N$ and applying Lemma \ref{lem:U} yields that
the induced IFSP according to eq. (\ref{ifsdefmulti}) consists of exactly $N^{d-1}$ 
similarities, each having contraction factor $\frac{1}{N}$. \\
Having this, applying Theorem \ref{thm:main2} yields property (P1),  
Corollary \ref{thm:main3} together with eq. (\ref{eq:calc_hausdorffdim}) property (P2), and 
it remains to prove (P3).   
Letting $\hat{\pc}_d$ denote the family of all $\mu \in \pc_d$ fulfilling 
$\mu^{\pi_{1:d-1}}=\lambda_{d-1}$ we 
obviously have $\pc_d^{\lambda_{d-1}} \subseteq \hat{\pc}_d$. 
According to \cite{GJT}, setting 
\begin{equation}
D_1(\mu,\nu)= \int_{[0,1]} \int_{[0,1]^{d-1}} \vert K_\mu(\mathbf{x},[0,y]) - K_\nu(\mathbf{x},[0,y]) \vert d\lambda_{d-1}(\mathbf{x})d\lambda(y)
\end{equation}
defines a metric on $\hat{\pc}_{d}$, the resulting metric space $(\hat{\pc}_{d},D_1)$ 
is complete and se\-pa\-rable, and convergence w.r.t. the metric $D_1$ implies uniform 
convergence of the corresponding copulas, which, in turn is equivalent to weak convergence of 
the $d$-stochastic measures (see, e.g., \cite{DS}). 
Using the fact that the probability spaces $([0,1]^{d-1},\mathcal{B}([0,1]^{d-1},\lambda_{d-1})$ and $([0,1],\mathcal{B}([0,1],\lambda)$ are isomorphic (see \cite{Royden,Walters}) 
and proceeding as in \cite{p6} yields the following:
firstly, the family of completely dependent measures in $\hat{\pc}_d$ is 
closed in $(\hat{\pc}_d,D_1)$; and, secondly, $\mathcal{V}_\tau$ is a contraction on 
the complete metric space $(\hat{\pc}_d,D_1)$. 
As a direct consequence, considering an arbitrary 
$\lambda_{d-1}$-$\lambda$-preserving transformation $g: [0,1]^{d-1} \rightarrow [0,1]$ 
and letting $\mu_g \in \hat{\pc}_d$ denote the corresponding 
completely dependent measure, it follows that $\mathcal{V}^n_\tau(\mu_g)$ is 
completely dependent too for every $n \in \mathbb{N}$ (since, as mentioned above, 
for arbitrary $(i_1,\ldots,i_{d-1}) \in \{1,\ldots,N\}^{d-1}$ there exists
exactly one $l_0 \in \{1,\ldots,N\}$ with $\tau((i_1,\ldots,i_{d-1},l_0))>0$). 
This altogether yields a sequence $(\mathcal{V}^n_\tau(\mu_g))_{n \in \mathbb{N}}$
of completely dependent measures in $\hat{\pc}_d$, which, using Banach's fixed point theorem
converges to some completely dependent measure $\nu \in \hat{\pc}_d$ w.r.t. $D_1$.
Since convergence w.r.t. $D_1$ implies weak convergence, $\nu=\mu_\tau^\star$ follows and 
we have shown that $\mu_\tau^\star$ is completely dependent, i.e., for 
$\mathbf{X}=(X_1,\ldots,X_d) \sim \mu_\tau^\star$ we have that 
$X_d$ is almost surely a function 
of $\mathbf{X}_{-d}$. Permuting the coordinates completes the proof.  
\end{proof}
\begin{example}\label{ex:sierpinski}
Consider $d=3,N=2$ and let $\sigma:=\tau \in \mathcal{U}_3^2$ be defined according to Theorem \ref{thm:sierpiski.multi}. Then the copula corresponding to 
$\mathcal{V}_\sigma(\lambda_3)$ coincides 
with the cube-copula $C^{\textrm{cube}}$ considered in \cite[Example 3.4. (3)]{MFT}. 
Figure \ref{fig2} depicts the supports of the 
proba\-bi\-lity measures $\mathcal{V}_\sigma^n(\lambda_3)$ for $n=1,2,3$ and $n=5$, Figure 
\ref{fig1} for $n=6$; notice that for every $n \in \mathbb{N}$ the probability measure 
$\mathcal{V}_\sigma^n(\lambda_3)$ coincides with the uniform distribution on $4^n$ cubes.
Looking at Figure \ref{fig2} it becomes clear that the support of $\mu_\sigma^\star$ 
coincides with a Sierpinski tetrahedron (a.k.a. tetrix, see \cite{Ts,W}). 
In other words, we have constructed a $3$-stochastic measure $\mu_\sigma^\star$ with the
following striking properties: the support of $\mu_\sigma^\star$ is a Sierpinski
tetrahedron, $\mu_\sigma^\star$ has uniform bivariate marginals, and $\mu_\sigma^\star$ is 
completely dependent in each direction. 
\end{example}

\begin{figure}[htbp]
    \centering
    \includegraphics[width=0.8\textwidth]{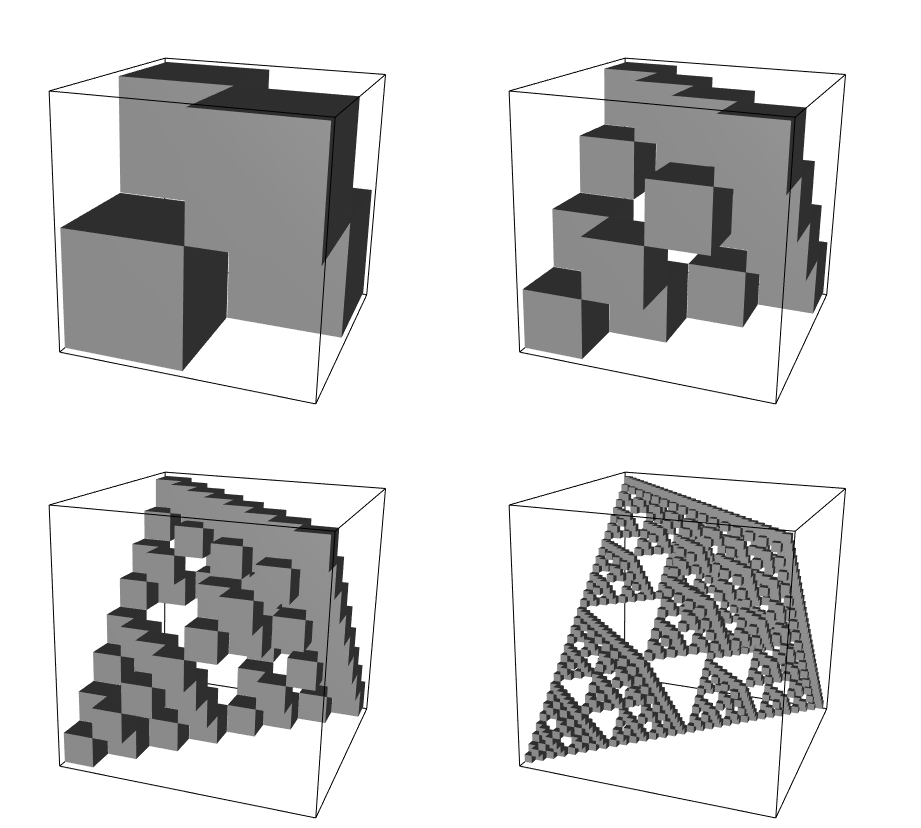}
    \caption{Support of $\mathcal{V}_\sigma^n(\lambda_3)$ as considered in Example 
    \ref{ex:sierpinski} for $n \in \{1,2,3,5\}$; Figure \ref{fig1} depicts the case $n=6$.}
    \label{fig2}
\end{figure}
We conclude this section with another three-dimensional example for 
the situation $m_1=m_2=m_3=N \geq 3$.
\begin{example}\label{ex:rotate}
Consider $d=3$ and $m_1=m_2=m_3=N \geq 3$, let 
$\epsilon$ denote a permutation of $\{1,\ldots, N\}$ such that every point $i \in \{1,\ldots,N\}$
has (minimal) period $N$, 
and define $\tau$ as the discrete uniform distribution on the finite set 
\begin{equation}\label{eq:rotation}
\{(i,\epsilon^{m}(i),m): \, i,m \in \{1,\ldots,N\}\} \subset \{1,\ldots,N\}^3.
\end{equation}
Then the intervals $E^j_k=[a^j_{k-1}, a^j_k]$ do not depend on $j \in \{1,2,3\}$ 
and we have $E^j_1=[0,\frac{1}{N}], E^j_2=[\frac{1}{N},\frac{2}{N}], \ldots,  
E^j_N=[\frac{N-1}{N},1]$. Moreover, for every pair $(i_1,i_2) \in \{1,\ldots,N\}^2$ there 
exists exactly one $j=j(i_1,i_2) \in \{1,\ldots,N\}$ with $\tau(i_1,i_2,j)>0$ so the construction 
of $\tau$ implies that the uniformity condition w.r.t. coordinate three holds.  
Since uniformity w.r.t. to the first and the second condition can be verified
analogously, $\tau \in \mathcal{U}_3^N$ follows. 
Proceeding as in the proof of Theorem \ref{thm:sierpiski.multi} we conclude that 
the measure $\mu_\tau^\star$ has self-similar support, that $\mu_\tau^\star \in \pc_3^{\lambda_2}$, 
and that for $(X_1,X_2,X_3) \sim \mu_\tau^\star$ we have that each variable is almost surely 
a function of the other two.  
\end{example}

\section{$d$-stochastic measures with uniform $(d-1)$-dimensional marginals and fractal support}
We start with the following simple example illustrating that for $d=3$ the class 
$\pc_d^{\lambda_{d-1}}$ contains elements whose support (is self-similar and) has
non-integer Hausdorff dimension. 
\begin{example}\label{ex:fractal}
Consider $d=N=3$ as well as the permutations $$\epsilon':=(231),\,\epsilon''=(312)$$ 
of $\{1,2,3\}$. Then for both permutations each point has period $3$, so Example \ref{ex:rotate}
implies that the two generalized transformation matrices $\tau',\tau''$ defined according
to eq. (\ref{eq:rotation}) fulfill $\tau',\tau'' \in \mathcal{U}_3^3$. 
Since $\mathcal{U}_3^3$ is convex it follows that $\tau=\frac{1}{2}(\tau'+\tau'')$ 
is an element of $\mathcal{U}_3^3$ too. It is straightforward to verify that $\tau$ assign mass 
$\frac{1}{18}$ to each of the points $(1,2,1), (2,3,1), (3,1,1), (1,3,1),(2,1,1),(3,2,1),
(1,3,2), (2,1,2), (3,2,2),(1,2,2)$ \newline and $(2,3,3),(3,1,2)$, and mass $\frac{1}{9}$
to each of $(1,1,3), (2,2,3), (3,3,3)$. According to Lemma \ref{lem:U}, the IFSP 
induced by $\tau$ only contains similarities, 
each with shrinking factor $\frac{1}{3}$, and 
the support $\textrm{supp}(\mu_\tau^\star)$ of $\mu_\tau^\star$ is a self-similar set 
whose Hausdorff dimension is the unique number $s$ fulfilling 
$15\,(\frac{1}{3})^s=1$, i.e., 
$$
\textrm{dim}_H(\mu_\tau^\star)= \tfrac{\log(15)}{\log(3)} \in (2,3).
$$
Applying Theorem \ref{thm:main2} shows $\mu_\tau^\star \in \pc_3^{\lambda_2}$, implying
that elements of $\pc_3^{\lambda_2}$ may have a support with non-integer Hausdorff dimension. 
Figure \ref{fig2} depicts the density of $\mathcal{V}_\tau^n(\lambda_3)$ for $n \in \{1,2,3\}$.
\end{example}
\begin{figure}[h!]
    \centering
    \includegraphics[width=0.7\textwidth]{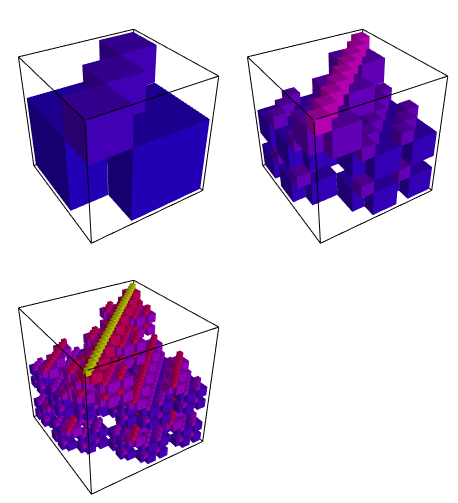}
    \caption{Density of $\mathcal{V}_\tau^n(\mu_{\lambda_3})$ for $n \in \{1,2,3\}$ with
    $\tau$ according to Example \ref{ex:fractal}; the color gradient going 
    from blue (low density) via red to yellow (high density).}
    \label{fig2}
\end{figure}
Using the idea of working with convex combinations of elements in $\mathcal{U}_d^N$ 
(as done in Example \ref{ex:fractal}) we 
now prove the first main result of this section: 
\begin{theorem}\label{thm:fractal.dim}
    For every $d \geq 3$ the set 
    \begin{equation}\label{eq:dense}
    \mathcal{D}_d:= \Big\{\mathrm{dim}_H(\mathrm{supp}(\mu)):\, \mu \in \pc_d^{\lambda_{d-1}}\Big\}    
    \end{equation}
   is dense in $[d-1,d]$.
\end{theorem}
\begin{proof}
Let $d \geq 3, N \geq 2$ be arbitrary but fixed and define     
$\tau$ as in Theorem \ref{thm:sierpiski.multi}, i.e.,    
\begin{equation*}
\tau(\mathbf{i})=
\begin{cases}
\frac{1}{N^{d-1}} & \text{if } \sum_{j=1}^d (i_j-1) \, \in \{0,N,2N,\ldots\}=:N\cdot \mathbb{N}_0\\
0 & \text{otherwise.}  
\end{cases}
\end{equation*}
Then we have $\tau \in \mathcal{U}_d^N$, the support of $\tau$ has cardinality $N^{d-1}$, and 
the induced IFSP consists of exactly $N^{d-1}$ similarities, each having shrinking factor 
$\frac{1}{N}$. Loosely speaking, our idea of proof
is to consider all possible permutations of $\tau$ and to work with convex combinations 
incorporating a growing number of elements in $\mathcal{U}_d^N$ which yields a growing number 
of similarities in the corresponding IFSP.  \\
To simplify notation we will let $\Sigma_N$ denote the family of all $N!$ permutations of 
the set $\{1,\ldots,N\}$. The set $(\Sigma_N)^d$ has cardinality $(N!)^d$, we will write 
elements of $(\Sigma_N)^d$ in the form $\mathbf{e} =(\epsilon_1,\ldots,\epsilon_d)$ and 
let $\mathbf{e}_1,\mathbf{e}_2,\ldots,\mathbf{e}_{(N!)^d}$ denote an arbitrary but fixed 
enumeration of $(\Sigma_N)^d$ fulfilling that $\mathbf{e}_1$ corresponds to the identity 
on $\{1,\ldots,N\}^d$. 
For every $\mathbf{e} =(\epsilon_1,\ldots,\epsilon_d) \in (\Sigma_N)^d$, setting 
$$
  \tau^\mathbf{e}((i_1,\ldots,i_d)):=\tau(\epsilon_1(i_1),\ldots,\epsilon_d(i_d)), 
$$
according to Lemma \ref{lem:U} we have $\tau^\mathbf{e} \in \mathcal{U}_d^N$, the cardinality
of the support $\mathcal{S}_{\tau^\mathbf{e}}$ of $\tau^\mathbf{e}$ fulfills 
$$
\#\mathcal{S}_{\tau^\mathbf{e}} = \# \mathcal{S}_{\tau}=N^{d-1}, 
$$
and the IFSP induced by $\tau^\mathbf{e}$ consists of $N^{d-1}$ similarities.  
For every $n \in \{1,\ldots,(N!)^d\}$ define $\sigma_n \in \mathcal{U}_d^N$ by 
\begin{equation}
    \sigma_n:= \tfrac{1}{n} \sum_{k=1}^n \tau^{\mathbf{e}_k}.
\end{equation}
Then we obviously have 
$$
N^{d-1} = \#\mathcal{S}_{\sigma_1} \leq \# \mathcal{S}_{\sigma_2} \leq \#\mathcal{S}_{\sigma_3}
\leq \cdots \leq \#\mathcal{S}_{\sigma_{(N!)^d}} = N^d
$$
and 
$$
\#\mathcal{S}_{\sigma_n} - \# \mathcal{S}_{\sigma_{n-1}} \leq N^{d-1} 
$$
holds for every $n \in \{2,\ldots,(N!)^d\}$. As a direct consequence, for every $k \in \{2,\ldots,N\}$ we can find some $n_k \in \{1,\ldots,(N!)^d\}$
fulfilling $$\# \mathcal{S}_{\sigma_{n_k}} \in [(k-1)N^{d-1},kN^{d-1}].$$
This shows that for the measure $\mu_{\sigma_{n_k}}^\star \in \pc_d^{\lambda_{d-1}}$ we have 
\begin{equation}\label{eq:final.lower.upper}
\tfrac{\log(k-1)}{\log(N)} + d-1 \leq \tfrac{\log(\#\mathcal{S}_{\sigma_{n_k}})}{\log(N)} = 
\textrm{dim}_H(\textrm{supp}(\mu_{\sigma_{n_k}^\star})) \leq 
\tfrac{\log(k)}{\log(N)} + d-1.
\end{equation}
Considering that $N \geq 2$ was arbitrary, that the set 
$$
\bigcup_{N=2}^\infty \Big\{\tfrac{\log(k)}{\log(N)}:\, k \in \{2,\ldots,N\}\Big\}
$$
is dense in $[0,1]$, and that each measure $\mu_{\sigma_n}^\star$ is an element 
of $\pc_d^{\lambda_{d-1}}$ completes the proof. 
\end{proof}
We conjecture that for every $d \geq 3$ even $\mathcal{D}_d=[d-1,d]$ holds,  
but we have not been able to prove this slightly stronger result. \\

After constructing various elements of $\pc_d^{\lambda_{d-1}}$ with strikingly pathological 
mass distributions, we conclude the paper by showing that such wild animals can be 
found `topologically everywhere'. 

\begin{theorem}\label{thm:fract.dense}
    For every $d\geq 3$ the family $\mathcal{F}_d^{\lambda_{d-1}}$ of elements in 
    $\pc_d^{\lambda_{d-1}}$ whose support has non-integer Hausdorff dimension is 
    dense in the metric space $(\pc_d^{\lambda_{d-1}},h)$.
\end{theorem}
\begin{proof}
  Let $\mu \in \pc_d^{\lambda_{d-1}}$ be arbitrary
  but fixed. It suffices to show that arbitrary close to $\mu$ there exists some 
  element of $\mathcal{F}_d^{\lambda_{d-1}}$. 
  Doing so we work with so-called checkerboard approximations as studied in \cite{Mi} and, contrary
  to before, now construct generalized transformation matrices $\tau$ from 
  elements $\mu$ in $\pc_d^{\lambda_{d-1}}$. 
  For every $N \geq 2$ and every 
  $\mathbf{i}=(i_1,\ldots,i_d) \in \{1,\ldots,N\}^d$ define the cube $C_\mathbf{i}^N$ by 
  $$
  C_\mathbf{i}^N=\bigtimes_{j=1}^d \Big[\tfrac{i_j-1}{N},\tfrac{i_j}{N}\Big],
  $$
  and set $\tau^N(\mathbf{i}):=\mu(C_\mathbf{i}^N)$. Considering that 
  $\mu$ is an element of $\pc_d^{\lambda_{d-1}}$ it follows that 
  $\tau$ is a generalized transformation matrix and that $\tau^N \in \mathcal{U}_d^N$ holds. 
  Notice that the measure 
  $\mathcal{V}_{\tau^N}(\lambda_d)$ coincides with the $N\times N \times\ldots \times N$ 
  checkerboard approximation of the measure $\mu$, which (again see \cite{Mi}) converges 
  to $\mu$ weakly for $N \rightarrow \infty$, i.e., 
  $\lim_{N \rightarrow \infty} h(\mathcal{V}_{\tau^N}(\lambda_d),\mu)=0$ holds. 
  Since weak convergence in $\mathcal{P}_d$ is equivalent to uniform convergence 
  of the corresponding copulas (see \cite{DS}), the latter implies 
  \begin{equation}\label{eq:final1}
      \lim_{N \rightarrow \infty} \underbrace{\sup_{\mathbf{x} \in [0,1]^d} 
  \Bigg\vert \mathcal{V}_{\tau^N}(\lambda_d)\left(\bigtimes_{i=1}^d [0,x_i]\right) -
  \mu\left(\bigtimes_{i=1}^d [0,x_i]\right) \Bigg\vert}_{=:\delta_N} =0
  \end{equation}
   Let $\nu$ denote any element of $\pc_d^{\lambda_{d-1}}$ with 
  $\textrm{dim}_H(\textrm{supp}(\nu)) \in (d-1,d)$. Then according to 
  Theorem \ref{thm:main2} the measure $\mathcal{V}_{\tau^N}(\nu)$ fulfills 
  $\mathcal{V}_{\tau^N}(\nu) \in \pc_d^{\lambda_{d-1}}$
  hence, using the fact that bi-Lipschitz transformations (like similarities) 
  preserve the Hausdorff dimension, in combination with countable 
  stability of the Hausdorff dimension (see \cite{Fal}) yields
  \begin{equation}
      \textrm{dim}_H(\textrm{supp}(\mathcal{V}_{\tau^N}(\nu))) = 
      \textrm{dim}_H(\textrm{supp}(\nu))\in (d-1,d).
  \end{equation}
  As final step we show that for every $\varepsilon>0$ there exists some $N_1 \in \mathbb{N}$
  such that for all $N \geq N_1$ 
  \begin{equation}\label{ineq:final}
      \sup_{\mathbf{x} \in [0,1]^d} 
  \Bigg\vert \mathcal{V}_{\tau^N}(\nu)\left(\bigtimes_{i=1}^d [0,x_i]\right) -
  \mu\left(\bigtimes_{i=1}^d [0,x_i]\right) \Bigg\vert < \varepsilon.
  \end{equation}
  According to eq. (\ref{eq:final1}) there exists some $N_0 \in \mathbb{N}$
  such that $\delta_N < \frac{\varepsilon}{2}$ for all $N \geq N_0$. 
  Choose $N_1 \geq N_0$ such that $\frac{d}{2N_1} < \frac{\varepsilon}{4}$ and consider 
  some $N \geq N_1$. It is straightforward to see that for every $\mathbf{x} \in [0,1]^d$
  there exists some 
  $\mathbf{z}^{\mathbf{x}} \in \mathcal{G}:=\{0,\frac{1}{N},\ldots,\frac{N-1}{N},1\}^d$ with 
  $\sum_{i=1}^d \vert x_i-z^\mathbf{x}_i \vert \leq \frac{d}{2N} < \frac{\varepsilon}{4}$.
  Hence, considering  
  $\mathcal{V}_{\tau^N}(\nu)(C_\mathbf{i}^N)=\mathcal{V}_{\tau^N}(\lambda_d)(C_\mathbf{i}^N)$
  for every $\mathbf{i} \in \{1,\ldots,N\}^d$, using Lipschitz continuity of copulas (see \cite{DS}) and setting 
  \begin{eqnarray*}
    \delta_N'(\mathbf{x})&:=&
  \Bigg\vert \mathcal{V}_{\tau^N}(\nu)\left(\bigtimes_{i=1}^d [0,x_i]\right) -
  \mathcal{V}_{\tau^N}(\lambda_d)\left(\bigtimes_{i=1}^d [0,x_i]\right) \Bigg\vert  
  \end{eqnarray*}
  using the triangle inequality we have 
  \begin{eqnarray*}
    \delta_N'(\mathbf{x}) &\leq& 
    \Bigg\vert \mathcal{V}_{\tau^N}(\nu)\left(\bigtimes_{i=1}^d [0,x_i]\right) -
  \mathcal{V}_{\tau^N}(\nu)\left(\bigtimes_{i=1}^d [0,z_i^\mathbf{x}]\right)  \Bigg\vert  
   \,+\, \delta_N'(\mathbf{z}^\mathbf{x})\\
  & & + \,\Bigg\vert \mathcal{V}_{\tau^N}(\lambda_d)\left(\bigtimes_{i=1}^d [0,z_i^\mathbf{x}]\right) -
  \mathcal{V}_{\tau^N}(\lambda_d)\left(\bigtimes_{i=1}^d [0,x_i]\right)  \Bigg\vert \\
  &\leq& \tfrac{\varepsilon}{4} + 0 + \tfrac{\varepsilon}{4} = \tfrac{\varepsilon}{2}.
  \end{eqnarray*}
  Applying the triangle inequality once more it follows that 
  for every $N \geq N_1$ we have 
  $\delta_N + \sup_{\mathbf{x} \in [0,1]^d} \delta_N'(\mathbf{x}) \leq \varepsilon$, which 
  shows ineq. (\ref{ineq:final}) and completes the proof. 
\end{proof}

\begin{remark}
Notice that in the proof of Theorem \ref{thm:fract.dense} we have shown a slightly stronger
result: for every $d\geq 3$ and every open, non-empty interval $I \subset (d-1,d)$ 
the family of all elements in $\pc_d^{\lambda_{d-1}}$ whose support has a Hausdorff dimension in the interval $I$ is dense in the metric space $(\pc_d^{\lambda_{d-1}},h)$. 
\end{remark}


\noindent \textbf{Acknowledgement}
The first and the third author gratefully acknowledge the support of the WISS 2025 project 
‘IDA-lab Salzburg’ (20204-WISS/225/197-2019542 and 20102-F1901166-KZP)











\end{document}